\documentclass{AIMS}

\usepackage{amsmath,paralist,graphicx}
\usepackage[colorlinks=true]{hyperref}
\hypersetup{urlcolor=blue, citecolor=red}

\usepackage{psfrag}
\usepackage{amsfonts}

\newtheorem{theorem}{Theorem}
\newtheorem{corollary}[theorem]{Corollary}
\newtheorem{lemma}[theorem]{Lemma}

\def\tst{{h}}
\def\RR{\mathbb{R}}
\def\CC{\mathbb{C}}
\def\domD{\mathcal{D}}
 
\title[Computing the modified energy]
      {On an asymptotic method for computing the modified energy 
       for symplectic methods}

\author[Per Christian Moan and Jitse Niesen]{}

\subjclass{Primary: 65P10; Secondary: 37J40, 37M15.}
\keywords{modified energy, symplectic integration, 
  Hamiltonian systems, Richardson extrapolation}

\email{pcmoan@gmail.com}
\email{jitse@maths.leeds.ac.uk}

\begin{document}
\maketitle

\centerline{\scshape Per Christian Moan}
\medskip
{\footnotesize
 \centerline{Centre of Mathematics for applications}
 \centerline{University of Oslo}
 \centerline{Norway}
}

\medskip

\centerline{\scshape Jitse Niesen}
\medskip
{\footnotesize
 \centerline{School of Mathematics}
 \centerline{University of Leeds}
 \centerline{United Kingdom}
}

\begin{abstract} 
  We revisit an algorithm by Skeel \textsl{et al.}~\cite{engle05med,
    skeel01pco} for computing the modified, or shadow, energy
  associated with the symplectic discretization of Hamiltonian
  systems. By rephrasing the algorithm as a Richardson extrapolation
  scheme arbitrary high order of accuracy is obtained, and provided
  error estimates show that it does capture the theoretical
  exponentially small drift associated with such
  discretizations. Several numerical examples illustrate the theory.
\end{abstract}

\section{Introduction} 

Numerical simulation of conservative differential equations requires
special care in order to avoid introducing non-conservative, or
non-physical truncation error effects. For Hamiltonian ODEs or
Euler--Lagrange equations originating from variational principles
there exists much evidence~\cite{hairer97lob, leimkuhler05shd,
  moan04oka, reich99bea} that the proper discretization scheme
should be \emph{symplectic}~\cite{hairer02gni, leimkuhler05shd,
  moan04oka, tupper05ean}. In the Hamiltonian case this can be
achieved by imposing special conditions on classical methods or by
methods based on generating functions~\cite{hairer02gni}. In the
variational formulation symplecticity is achieved by discretizing the
action integral and carrying out a discrete
variation~\cite{marsden01dma}. In some cases these formulations and
methods turn out to be equivalent by the Legendre
transformation~\cite{jay07bcr, marsden01dma}.

Focusing on the Hamiltonian side, symplecticity implies that the
trajectory produced by the numerical algorithm \emph{is the exact
  solution}~\cite{moan06orm} of another, non-autonomous ``modified''
Hamiltonian system close to the original one. Various stability
results for Hamiltonian ODEs then apply, leading to an understanding
of the dynamics of such discretizations schemes~\cite{hairer97lob,
  leimkuhler05shd, moan04oka, shang99kto}. Early results on modified
equations focused on the autonomous part~\cite{benettin94ohi,
  calvo94mef, hairer97lob, moan06ome, reich99bea} and established that
its flow is exponentially close to the numerical trajectory. This work
was motivated by the bounded error in energy observed in simulations
with symplectic schemes. The early results are contained in the newer
results since the time-dependent part is exponentially small due to
analyticity. Despite its smallness the non-autonomous term excites
instabilities through resonances, one consequence being a drift in the
modified energy. In simulations requiring millions of steps such as in
molecular dynamics~\cite{leimkuhler05shd, skeel97fos} and celestial
mechanics~\cite{wisdom92smf} these effects become significant and it
becomes important to understand and control them.

Constructing the modified Hamiltonian is equivalent to evaluating many
terms in the Baker--Campbell--Hausdorff formula, or its continuous
analogue~\cite{moan06ome}, a combinatorially complicated task possible
only for small systems and to a low order of accuracy.  Recently, Skeel
and coworkers~\cite{engle05med, skeel01pco} devised a method for
numerically computing the value of the modified Hamiltonian along the
numerical trajectory, thus allowing us to track the possible drift in
the modified energy. In this paper we simplify this method, possibly
at the cost of extra storage, and provide exponentially small error
bounds when it is applied in the asymptotic regime. It is then used to
verify, and justify the theory of modified equation on several test
equations and methods.

\section{Modified equations}

As alluded to in the introduction, the numerical solution of an ODE
$y' = f(y)$ is interpolated by the exact solution of a modified ODE
$\overline{y}' = f(\overline{y},t)$. The modified equation is
non-autonomous, but the non-autonomous part is exponentially small in
the step size. More precisely, given an analytic vector field $f$ and
a one-step method defined by an analytic mapping~$\Psi_{\tst,f}$,
there exists an analytic vector field $\overline f(y,t)$,
$\tst$-periodic and analytic in $t$, whose exact flow exactly
interpolates the numerical trajectory $\{x_n\}$,
$x_{n+1}=\Psi_{\tst,f}(x_n)$. The construction in~\cite{moan06orm}
starts by constructing a modified vector field $\tilde f(y,t)$ which
is only $C^\infty$ in $t$ whose flow interpolates $\{x_n\}$. This
vector field is then transformed by a time-dependent coordinate
transformation into a vector field $\overline f$ analytic in $t$.

The domain of analyticity of $\tilde f$ plays an important role in the
analysis, and we have found it useful to assume that $\tilde f$ is
analytic for all $y$ in a domain of the form
\[
\domD_y := \bigcup_{t>0}\{z\in \CC^d: |\tilde y(t)-z|_\infty < \tilde r_y\}
= \bigcup_{t>0}\{z\in \CC^d: |\Im(\tilde y(t)-z)|_\infty < \tilde r_y\}
\]
for some $\tilde r_y>0$, where $\tilde y(t) =
\phi_{t,\smash{\tilde{f}}}(y_0)$ is the trajectory of the smooth
modified vector field $\tilde f$. This domain is typically smaller
than the domain of analyticity of $f$, and depends on the numerical
method. In the following we will use the sup-norm
$\|f\|_{\domD}=\sup_{z\in \domD_y} |f(z)|_\infty$. With these
definitions the main result of~\cite{moan06orm} in the limit
$\tst\rightarrow 0$ can be formulated as

\begin{theorem}\label{BEA} 
  Let $\Psi_{\tst,f}$ be a one-step method applied to the analytic
  vector field~$f$, and $y_{n+1}=\Psi_{\tst,f}(y_n)$ be the
  approximations obtained by iterating $\Psi$. Then there exists a
  modified vector field $\overline f(\overline y,t) = f(\overline y) +
  r_1(\overline y) + r_2(\overline y,t)$ which is $h$-periodic in $t$
  and analytic in $(y,t)\in \overline \domD'$ such that its exact flow
  satisfies $\overline y(n\tst)=\Phi_{n\tst,\overline{f}}(y_0) = x_n$.
  In the limit $\tst \rightarrow 0$ we have the estimates
  \begin{align*}
    \|\overline f\|_{\overline \domD'} &\leq \frac{2}{1-\eta} \|f\|_{\domD}
    \\
    \|r_2\|_{\overline \domD'} & =
    \mathcal{O}\left(\frac{\|f\|_{\domD}}{\tst}
      \exp\left(-\eta\frac{2\pi\delta}{\|f\|_{\domD}e \tst } \right) 
    \right) 
  \end{align*}
  for $0<\eta<1$, $0< \delta < \tilde r_y$. The domain of analyticity
  of the modified vector field is 
  \[
  \overline \domD' = \left\{ (z,\tau) \in \CC^d\times \CC :
  |\Im(z-\overline y(t))|_\infty< \tilde r_y-\delta,\,
  |\Im(\tau-t)| < \frac{\eta \delta}{\|f\|_{\domD} e} \right\}, 
  \]
  and the norm $\|\,\cdot\,\|_{\domD}$ is defined by $\|\overline
  f\|_{\overline \domD'}=\sup_{(z,\tau)\in \overline \domD'}|\overline
  f(z,\tau)|_\infty$.
\end{theorem} 

For a Hamiltonian vector field $f$ and a symplectic numerical
method~\cite{hairer97lob, hairer02gni, reich99bea}, the modified
vector field $\overline f$ is also Hamiltonian~\cite{calvo94mef,
  hairer02gni, leimkuhler05shd}, with Hamiltonian $\overline H = H(y)
+ G_1(y) + G_2(y,t)$ where $H$, $G_1$ and~$G_2$ are the Hamiltonians
corresponding to the vector fields $f$, $r_1$ and~$r_2$,
respectively. The change in the \emph{modified energy} along the
numerical trajectory therefore satisfies
\[
\frac{d}{dt}\overline H
= \{\overline H,\overline H\}+\frac{\partial}{\partial t}\overline H
= \frac{\partial}{\partial t}G_2,
\]
where $\{F,G\} := \sum_{j=1}^m F_{q_j}G_{p_j} - F_{p_j}G_{q_j}$ is the
Poisson bracket.  By Theorem \ref{BEA} this drift is very small for
small $\tst$, thus motivating symplectic methods.

\section{The method of Skeel \textsl{et al.} for 
  computing the modified energy}

Skeel and coworkers~\cite{engle05med, skeel01pco} found an ingenious
way of evaluating the modified energy $\overline H$ at the points
$\{x_n\}$ for discretizations based on splittings~\cite{blanes02psr}.
Suppose we have an Hamiltonian given by $H = \frac12 p^T M^{-1} p +
U(q)$. An explicit splitting algorithm with step size $h$ is given by
\begin{equation}
  \label{splittmeth}
  \begin{split}
    & \text{for } n = 0,1,2,... \\
    & \quad \hat p_0 = p_n,\quad \hat q_0=q_n \\
    & \quad \text{for } s = 1:S \\
    & \quad\quad \hat p_s = \hat p_{s-1} - \tst a_s U_q(\hat q_{s-1}) \\
    & \quad\quad \hat q_s = \hat q_{s-1} + \tst b_s M^{-1}\hat p_{s} \\
    & \quad \text{end} \\
    & \quad p_{n+1} = \hat p_S, \quad q_{n+1}=\hat q_S \\
    & \text{end} 
  \end{split}
\end{equation}
leading to approximations $p_{n+1} = \hat p_S$, $q_{n+1} = \hat q_S$
when $\hat q_0 = q_n$, $\hat p_0 = p_n$ at $t_n = n\tst$. By choosing
the coefficients $a_s$,~$b_s$ appropriately, a method of arbitrary high
order can be found. The modified Hamiltonian can be found by
representing the inner loop of (\ref{splittmeth}) as a concatenation
of exponential operators~\cite{hairer02gni}
\begin{multline*}
  \Psi_{h,f}(p,q) = 
\exp(-ha_1 U_q\partial_p)(p,q)\exp(hb_1 M^{-1}p\partial_q)
 \cdots \\
\exp(-ha_S U_q\partial_p)\exp(hb_S M^{-1}p\partial_q)  , 
\end{multline*}
whereby the Baker--Campbell--Hausdorff (BCH) formula is used to find an
expression so that $\Psi_{h,f}(x) \simeq \exp(h\overline f\partial)(x)$.

The approach of Skeel \textsl{et al.} for computing values of the
modified energy is to append one scalar equation to the numerical
integrator,
\begin{equation}
  \label{splittmeth2}
  \begin{split}
    & \text{for } n = 0,1,2,\dots \\
    & \quad \hat p_0 = p_n, \quad \hat q_0=q_n, 
    \quad \hat\beta_0 = \beta_n \\
    & \quad \text{for } s = 1:S \\
    & \quad\quad \hat p_s = \hat p_{s-1} - \tst a_s U_q(\hat q_{s-1}) \\
    & \quad\quad \hat\beta_s = \hat\beta_{s-1} 
    - \tst a_s(\hat q_{s-1}^T U_q(\hat q_{s-1}) + 2U(\hat q_{s-1})) \\
    & \quad\quad \hat q_s = \hat q_{s-1} + \tst b_s M^{-1}\hat p_{s} \\
    & \quad \text{end} \\
    & \quad p_{n+1} = \hat p_S, \quad q_{n+1} = \hat q_S,
    \quad \beta_{n+1} = \hat\beta_S \\
    & \text{end}
  \end{split}
\end{equation}
where $\beta_0 = 0$. To understand how the modified energy can be
recovered from $\{p_n,q_n,\beta_n\}$, note that by Theorem~\ref{BEA}
the numerical trajectory $(p_n,q_n)$ is exactly interpolated by the
flow of a Hamiltonian $\overline H$. The
discretization~(\ref{splittmeth2}) is the discretization of $H_\alpha
= \alpha^2H(\alpha^{-1}p, \alpha^{-1}q)$ where $\beta$ is conjugate to
$\alpha$, the so-called \emph{homogeneous extension} of $H(p,q)$. The
discovery in~\cite{skeel01pco} rests on the fact that homogeneous
extension is a Lie algebra homeomorphism. Thus, since $\overline H$ is
constructed by Poisson brackets as in the BCH formula, the modified
Hamiltonian for (\ref{splittmeth2}), $\overline H_\alpha$, is the
homogeneous extension of $\overline H$, hence the trajectory generated
by (\ref{splittmeth2}) is interpolated by
\begin{align*}
\overline q' &= \overline H_{\overline p}(\overline p,\overline q,t), \\
\overline p' &= -\overline H_{\overline q}(\overline p,\overline q,t), \\
\overline \beta' 
&= \overline q^T \overline H_{\overline q}(\overline p,\overline q,t)
+ \overline p^T\overline H_{\overline p}(\overline p,\overline q,t)
- 2\overline H(\overline p,\overline q,t),
\end{align*}
from which 
\begin{equation}
  \overline H = \frac12 (\overline q^T\overline H_{\overline q}
  + \overline p^T\overline H_{\overline p}-\overline \beta')
  = \frac12 (-\overline q^T \overline p' + \overline p^T \overline q'
  - \overline \beta'). \label{skModEn}
\end{equation} The equation for  $\alpha$ is removed from (\ref{skModEn}) since $\overline H_\alpha$ does not depend on $\beta$, the conjugate variable of $\alpha$,
and hence $\alpha'=0$ which is solved exactly by the methods we are considering.

Thus the value of $\overline H$ can be computed by finding the
derivatives of the interpolating trajectory (which are not known since
we do not have $\overline H_p,\overline H_q$). In~\cite{engle05med,
  skeel01pco} estimates of the derivatives are computed using backward difference formulas and 
interpolating polynomials with stored values of $\{p_n,q_n,\beta_n\}$. 
These polynomials can be precomputed, but unfortunately the required
expressions are very large, and they only provide expressions up to order 24. Their method does however have an advantage in requiring less stored values than one based on centered differences, which might be important if the modified energy is part of the simulation~\cite{engle05med}.

\section{Richardson extrapolation}
 
Our suggestion is to use Richardson extrapolation in order to avoid
the large expressions that arise in the method described in the
previous section.

First consider the use of Richardson extrapolation to find the
derivative of a function, say~$y$, at zero given the function values
on a grid. We define the central difference approximations
\[
T_{j,1} = \frac{y(jh) - y(-jh)}{2jh}, \quad j=1,\ldots,
\]
and compute the Richardson table entries
\[
T_{j,k+1} = T_{j,k} + \frac{T_{j,k}-T_{j-1,k}}{(1-k/j)^2-1},
\quad k=1,\ldots,j-1.
\]
We then have by standard results that $T_{j,j} = y'(0) +
\mathcal{O}(\tst^{2j})$. In fact, it is straightforward to prove by
induction that the $T_{j,k}$ satisfy
\[
T_{j,k} = \sum_{i=1}^k \frac{2 \, (-1)^{i+1} \, (j)_k^2}%
{(i-1)! \, (k-i)! \, (2j-k+i)_k \, (j-k+i)} T_{j-k+i,1}
\]
where the Pochhammer symbol denotes the falling factorial: 
\[
(n)_k = n (n-1) (n-2) \ldots (n-k+1) = \frac{n!}{k!}.
\]
It follows that the diagonal entries in the Richardson table are given
by $T_{m,m} = D_my(0)$ where $D_m(y)$ denotes the central-difference
approximation to the derivative~$y'(0)$ using $2m$~points, defined by
\begin{equation}
D_my(0) = \sum_{j=1}^m \frac{(-1)^{j+1} (m!)^2}%
  {j\tst(m-j)!(m+j)!} \bigl( y(j\tst) - [y(-j\tst) \bigr).  \label{cdiff}
\end{equation}
This approximation satisfies~$D_my(0) = y'(0) + \mathcal{O}(h^{2m})$.
By choosing the index~$m$ appropriately, it is possible to find an
exponentially accurate approximation for the derivative of analytic
functions, as stated in the following lemma.

\begin{lemma}\label{extlem} 
  Let $y(t)$ be analytic in $\{t\in \CC:|\Im(t)|<\rho\}$, then there
  exists an $m^*$ (which depends on~$h$) and a constant $C_1>0$ such that
  \[
  |y'(0)-D_{m^*}y(0)| 
  \leq C_1 \frac{\rho \exp\left(-\frac{\pi\rho}{\tst}\right)}{\tst^2} 
  \|y\|_\rho
  \] 
where $\|y\|_\rho=\sup_{|\Im(\tau )|< \rho} |y(\tau)|_\infty$.
\end{lemma} 
The proof of this result and other results are found in the Appendix.

\section{Computing the modified energy using 
  Richardson extrapolation}
 
Returning to the computation of the modified energy~\eqref{skModEn},
the derivatives in this formula can be computed with Richardson
extrapolation using the stored values of~$\{p_n,q_n,\beta_n\}$. To
compute the modified energy at $t = n\tst$, we define the central
difference approximation
\begin{align*}
  T_{j,1}^n &= \frac 12 \left(
    -\overline q^T(n\tst) \,
    \frac{\overline p((n+j)\tst)-\overline p((n-j)\tst)}{2j\tst}\right. \\
  &\qquad + \left.
    \overline p^T(n\tst) \,
    \frac{\overline q((n+j)\tst)-\overline q((n-j)\tst)}{2j\tst}
    -\frac{\overline \beta((n+j)\tst)-\overline \beta((n-j)\tst)}{2j\tst}\right) \\
  &= \frac 12 \left( -q^T_n \frac{ p_{n+j}-p_{n-j}}{2j\tst}
    + p^T_n \frac{q_{n+j}-q_{n-j}}{2j\tst} 
    - \frac{\beta_{n+j}-\beta_{n-j}}{2j\tst} \right)
\end{align*}
for $j=1,\ldots$ and then compute the Richardson table entries as before:
\begin{equation}
  T_{j,k+1}^n = T_{j,k}^n + \frac{T_{j,k}^n-T_{j-1,k}^n}{(1-k/j)^2-1},
  \quad k=1,\ldots,j-1, \label{Rich}
\end{equation}
The expression~$T_{j,j}^n$ is a convenient way of computing the
approximations and in addition it gives a way of estimating the error
in the approximation $|\overline H -T_{j-1,j-1}| \approx
|T_{j,j}-T_{j-1,j-1}|$, which is useful for finding a stopping
criterion for the extrapolation process.

We mention in passing that accurate values of $\overline H$ might be
obtained using Fourier series as well~\cite{benettin99fhp}, however
such methods seem most useful for quasi-periodic motions or scattering
problems, while the approach taken here seems suitable for a broader
range of problems.
 
The following corollary follows from Lemma~\ref{extlem}.

\begin{corollary}\label{coro1} 
  Let $p_n,q_n,\beta_n$ be given by the numerical scheme then there
  exists an $m^*$ such that
  \[
  |\overline H(p_n,q_n,t_n)-T_{m^*,m^*}^n| \leq 
  \frac{C_1}2 \frac{\rho\exp\left(-\frac{\pi\rho}{\tst}\right)}{\tst^2}  
  \bigl( |q_n|_1 \|\overline p\|_\rho + |p_n|_1 \|\overline q\|_\rho 
  + \|\overline \beta\|_\rho \bigr),
  \]
  where $\rho$ is such that the interpolating trajectory $(\overline
  p(t),\overline q(t),\overline \beta(t))$ is analytic for $|\Im(t)| <
  \rho$.
\end{corollary} 

In Corollary \ref{coro1} the parameter $\rho$ related to the domain of
analyticity of $\overline y(t)$ is undetermined. The following
existence lemma will be useful for bounding $\rho$.

\begin{lemma}[Domain of analyticity of the solution] \label{lemmaana}
  Let $g(y,t)$ be an analytic vector field on the domain
  \begin{align*}
    (z,\tau) & \in \domD_y\times\domD_t \\
    \domD_y & =\{z\in \CC^d : |z-x_n|_\infty < r_y\} \\
    \domD_t & =\{t\in \CC^d : |\Im(t)| < r_t\}
  \end{align*}
  Then $y(t)$ satisfying $y'=g(y,t)$, $y(0)=x_n\in \RR^d$ is an
  analytic function of $t$ on the domain
  \[
  \domD = \left\{t\in \CC:|t| < \min \left( \frac{r_y}{\|g\|_r}, r_t \right) \right\},
  \]
  where $\|g\|_r:=\sup_{(z,t)\in \domD_y\times \domD_t}|g(z,t)|_\infty$.
\end{lemma}

We can now combine the estimates of Corollary~\ref{coro1} and
Lemma~\ref{lemmaana} to determine a bound on $\rho$, and hence on the
error in the numerically computed modified energy.
\begin{theorem}[Numerical modified energy]\label{theomod} 
  Let $H(p,q)$ be analytic in its arguments, and let $p_n,q_n,\beta_n$
  be computed by the algorithm (\ref{splittmeth2}). Then for each $n$
  there exists an~$m^*$ such that we have the error bound
  \[
  |\overline H(p_n,q_n,t_n)-T_{m^*,m^*}^n| 
  \leq \frac{C_1}{2\tst^2} 
  \exp\left(-C_2 \frac{\delta}{\tst \|f\|_\domD}\right)
  \bigl( |q_n| \|\overline p\|_\rho + |p_n| \|\overline q\|_\rho
  + \|\overline\beta\|_\rho),
  \] 
  where $C_2<2.14707$ (and $\rho=0.6835\delta/\|f\|_\domD$).
\end{theorem}
The error bound in Theorem \ref{theomod} shows that we are able to
track the modified energy exponentially accurately. Moreover the bound
displays the same dependency on the parameters $h$, $\tilde r_y$ and
$\|f\|_\domD$ as Theorem \ref{BEA}. 

The bounding-constant~$C_2$ is however smaller than the $2\pi/e$ found
in the proof of Theorem~\ref{BEA}. It is unclear to us if this is due
to the proof techniques applied, or if it is an actual weakness of the
extrapolation method when applied to estimate the derivatives and thus
the modified energy.

\section{Numerical experiments}

We have implemented the extrapolation algorithm using the Arprec
multiple-precision library in order to avoid pollution by round-off
errors and to be able to verify the theory to high accuracy.%
\footnote{The Arprec library is available from
  \url{http://crd.lbl.gov/~dhbailey/mpdist/}. The C++ source code for our
  experiments can be downloaded from
  \url{http://www1.maths.leeds.ac.uk/~jitse/software.html}}  
We set the precision to 120~digits. Most experiments are done using
the standard St\"ormer--Verlet scheme (also known as the leap frog
scheme). 
All the experiments were also repeated with two fourth-order splitting schemes
to check for dependence on splitting scheme coefficients. No noteworthy dependence was found,
and we only present these results for the Kepler experiment.


An early experiment verifying the exponentially small drift in
modified energy was done by Benettin and
Giorgilli~\cite{benettin94ohi} who used a Hamiltonian of the form $H =
\frac12(p_1^2+p_2^2) + U(q_1^2+q_2^2)$ with the potential function~$U$
vanishing fast as its argument becomes large. In this case, the
exponentially small effects can be observed directly because methods
of the form~(\ref{splittmeth}) preserve the energy~$H$ exactly when
$U$ is identically zero. To carry out this experiment the initial
values $y_0=(p_1(0),p_2(0),q_1(0),q_2(0))$ are then chosen so that $U$
vanishes and that the trajectory passes close to $(q_1,q_2)=(0,0)$
before ending at some point $y_T=(p_1(T),p_2(T),q_1(T),q_2(T))$ where
again $U$ vanishes. Carrying out that simulation the difference
$|H(y_0)-H(y_T)|$ is observed to be ${\mathcal O}(\exp(-C/h))$ where
$C$ is some unspecified constant.

We repeated this experiment using our method, and found that in this
case it had \emph{zero error} so the experiments we consider will not
have this type of Hamiltonian.  This matter warrants further
investigations, but we have not pursued these in this paper.

\subsection{The pendulum}

\begin{figure}
  \centering
  \includegraphics[width=12cm]{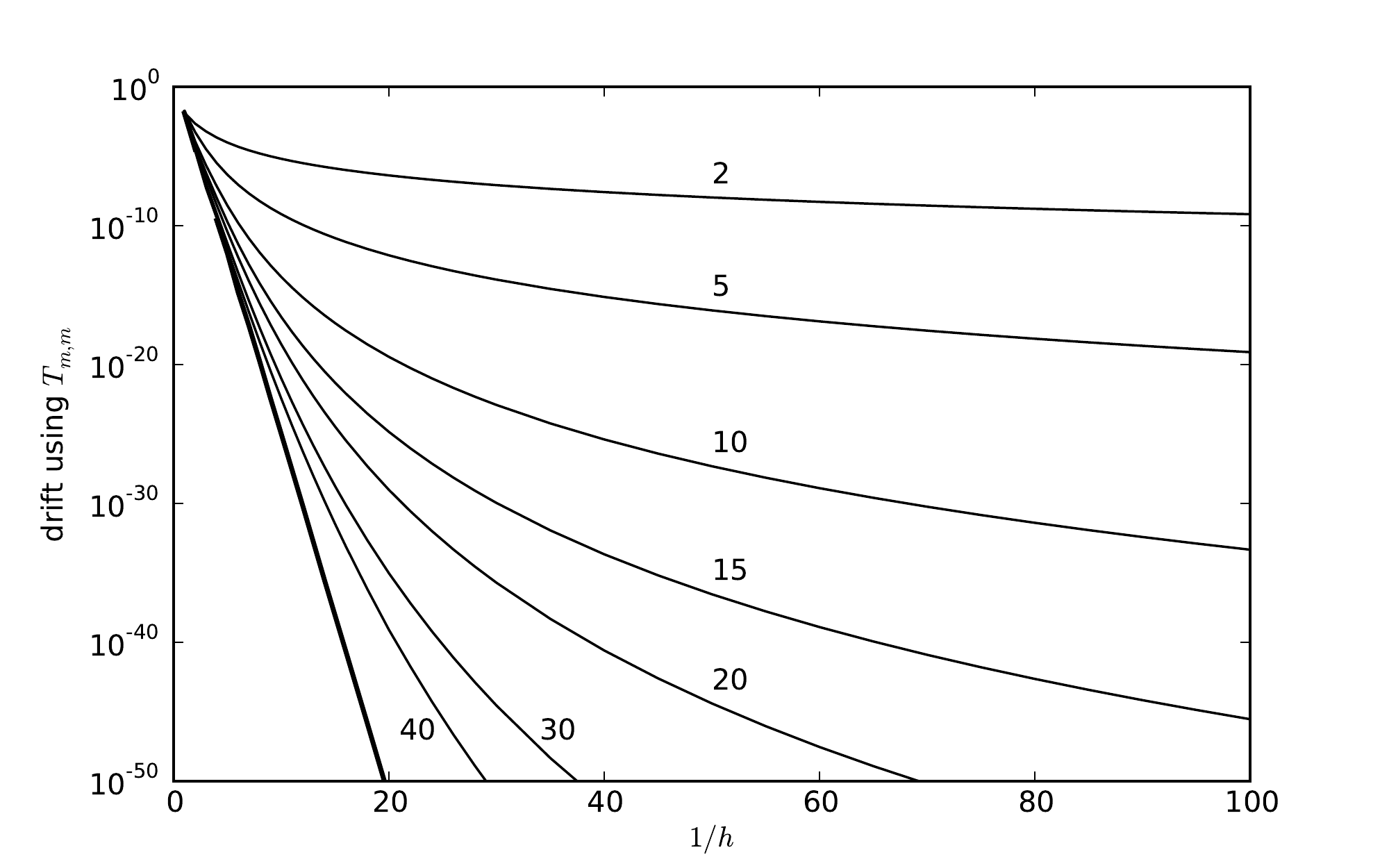}
  \caption{Drift in modified energy for the pendulum as a function of
    step size~$h$ for various values of~$m$. The thick line indicates
    the limit.}
  \label{fig3}
\end{figure} 

In this experiment we apply the St\"ormer--Verlet method to the
pendulum, which has Hamiltonian $H = p^2/2 - \cos(q)$, integrated over
the time interval $[0,100]$. Figure~\ref{fig3} reports the drift in
the modified energy computed using~$T_{m,m}$ for $m=2, 5, 10, 15, 20,
30, 40$. Here, and in the other plots, the drift is defined as the
difference between the maximum of the modified energy over the
integration interval and its minimum. The figure suggests that the
approximations~$T_{m,m}$ converge for this problem. The limit is
indicated by the thicker line in the left part of the plot, which
shows that the drift follows the $\exp(-c/h)$ behaviour predicted by
the theory.

\begin{figure}
  \centering
  \includegraphics[width=12cm]{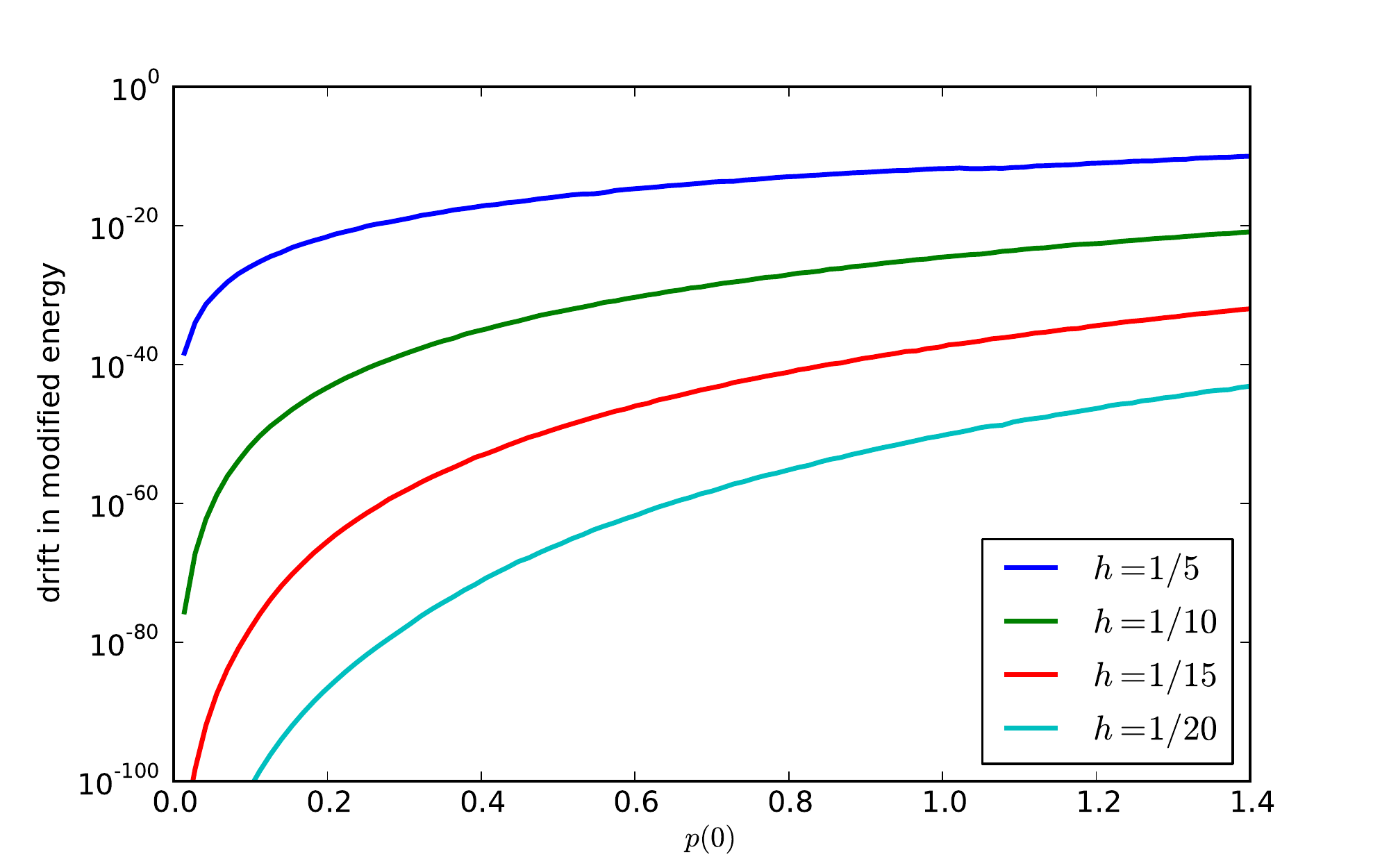}
  \caption{Drift in modified energy for the pendulum as a function of
    the initial momentum~$p(0)$.}
  \label{fig4}
\end{figure} 

The initial value for the experiment reported in Figure~\ref{fig3} is
$q(0)=0$ and $p(0) = 1$. Next we study the effect of varying the
initial condition. The result is shown in Figure~\ref{fig4}. The
St\"ormer--Verlet method shows improved energy preservation near the
equilibrium point, revealing the $\exp(-c/h\|f\|_\domD)$ dependency on
step size and on~$\|f\|_\domD$ which decreases as $p\rightarrow 0$.

\subsection{Kepler problem}

The Kepler problem for one particle in a central force field is given
by the Hamiltonian
\[
H = \frac 1 2 (p_1^2+p_2^2)-\frac 1 {\sqrt{q_1^2+q_2^2}}.
\]
We integrate over the interval~$[0,100]$ starting from the point
\begin{align*}
  p_1 &= 0,   & q_1 &= 1-ecc, \\
  p_2 &= \sqrt{\frac{1+ecc}{1-ecc}}, &q_2 &= 0.
\end{align*}
where $0\leq ecc < 1$ is the eccentricity of the orbit.

\begin{figure}
  \includegraphics[width=0.5\linewidth]{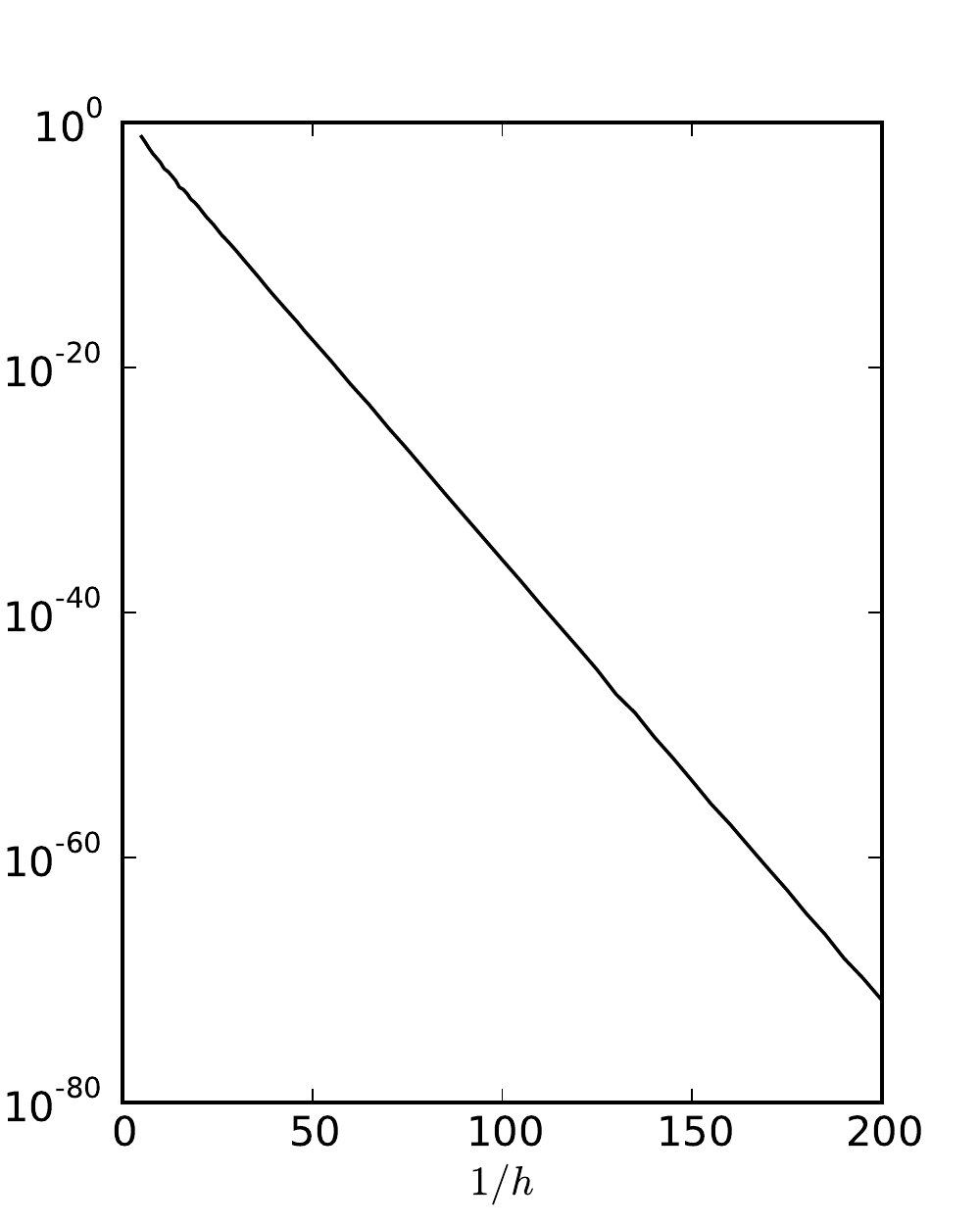}%
  \includegraphics[width=0.5\linewidth]{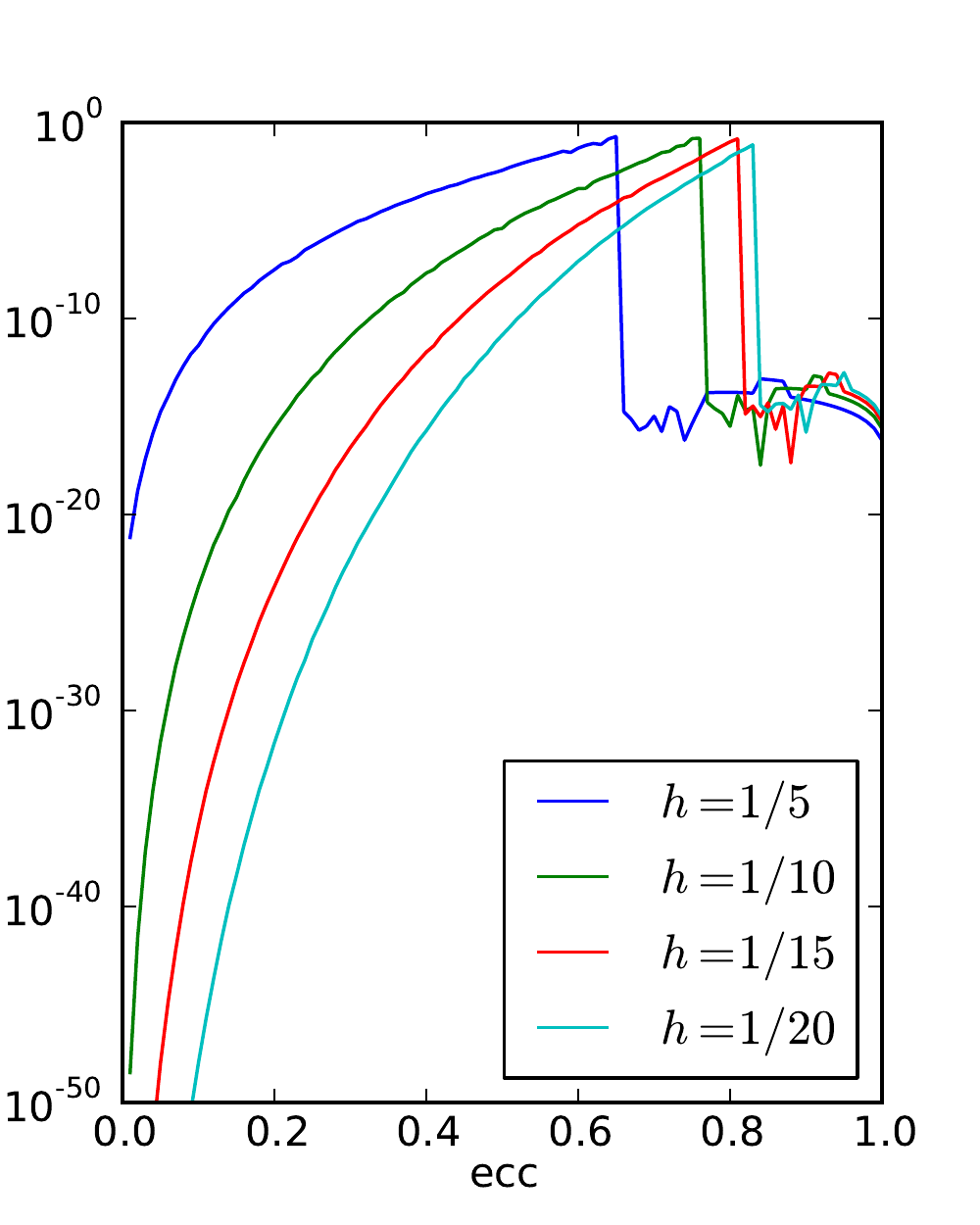}
  \caption{Drift in modified energy for the Kepler problem as a
    function of step size~$h$ for eccentricity $ecc=0.6$ (left plot),
    and as a function of eccentricity (right plot).}
  \label{figkepler}
\end{figure}

The left plot of Figure~\ref{figkepler} shows the theoretical
$\exp(-c/h)$ behavior. In contrast with the pendulum, for this problem
the $T_{m,m}$ do not converge as $m \to \infty$, but the sequence has
to be truncated at a suitably chosen point. To find the optimal~$m$,
we approximate the error in the $m$-th estimate as
\begin{equation}
|\overline H - T_{m,m}| \approx |T_{m,m} - T_{m-1,m-1}|. \label{errest}
\end{equation}
We compute this estimate for $m=2,3,\ldots,200$ and select the value
of~$m$ for which the estimated error is minimized. This procedure
recovers the expected exponential behaviour.

The right plot of Figure~\ref{figkepler} shows the dependence of the
drift in the modified energy on the eccentricity of the orbit. Almost
circular orbits with a low eccentricity show much better preservation
of the energy than highly elliptical orbits. An instability occurs at
a critical eccentricity which depends on the step size.  This can be
explained by the fact the the topology of the energy levels of the
modified energy changes with $h$, thus leading to unbounded
trajectories and instability.

We used the second-order St\"ormer--Verlet method to produce
Figure~\ref{figkepler}.  We ran the experiments again with two
fourth-order splitting methods: Yoshida's scheme based on
extrapolation~\cite{yoshida90coh} and a fourth-order scheme due to
Blanes and Moan~\cite{blanes02psr}. The last scheme is optimized for
problems of the type we have considered. It has very small error
coefficients at the cost of many stages, leading to coordinate errors
which are typically three orders of magnitude smaller than Yoshida's
method at the same computational cost. The plots for the drift in
modified energy of both Yoshida's method and the Blanes--Moan method
look the same as for the St\"ormer--Verlet method. In particular, the
constant~$c$ in $\exp(-c/h)$ is the same. However, the drift in the
modified energy for the St\"ormer--Verlet method is approximately a
factor of three smaller than Yoshida's method and a factor of four
smaller than the method of Blanes and Moan.

\begin{figure}
  \includegraphics[width=0.5\linewidth]{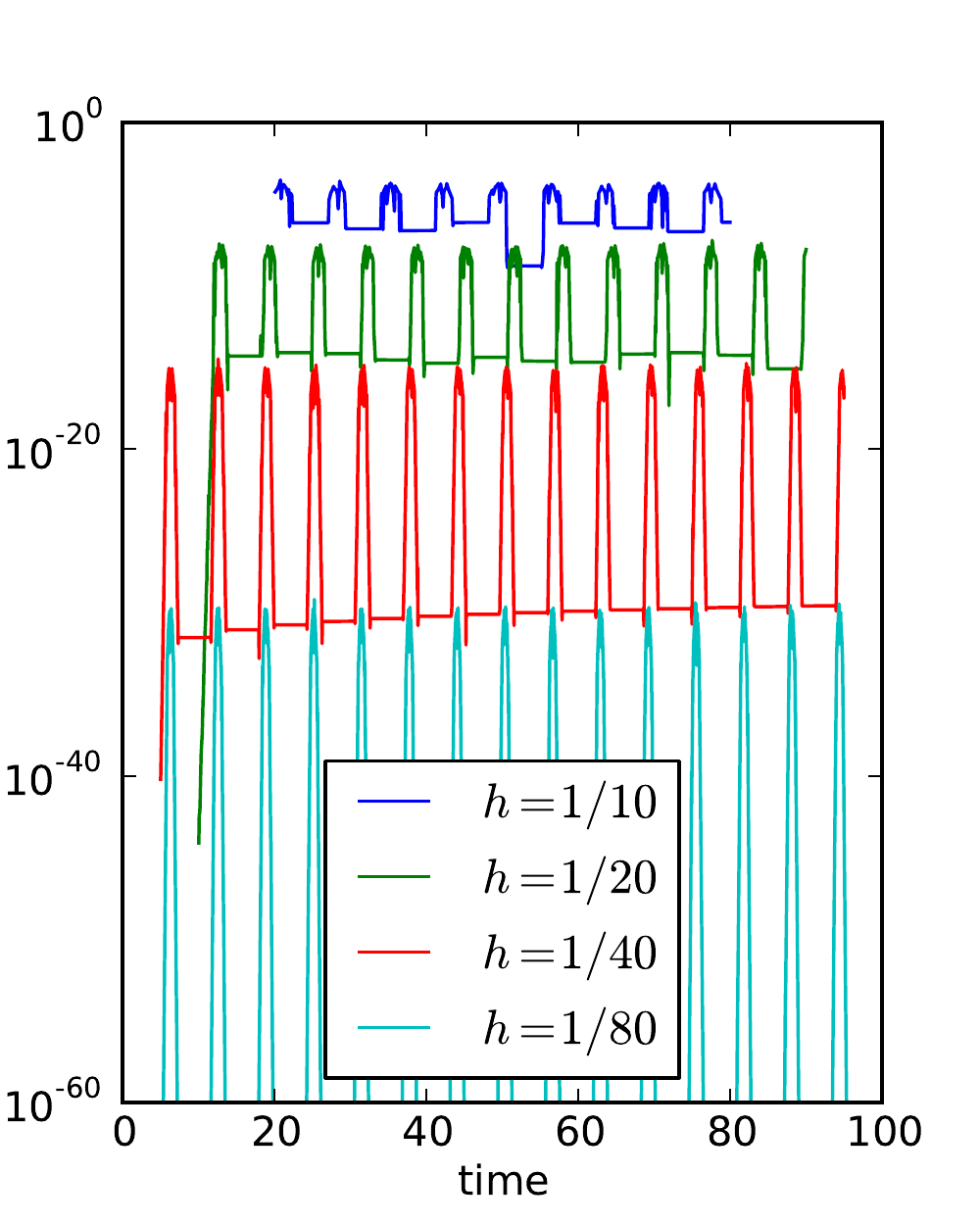}%
  \includegraphics[width=0.5\linewidth]{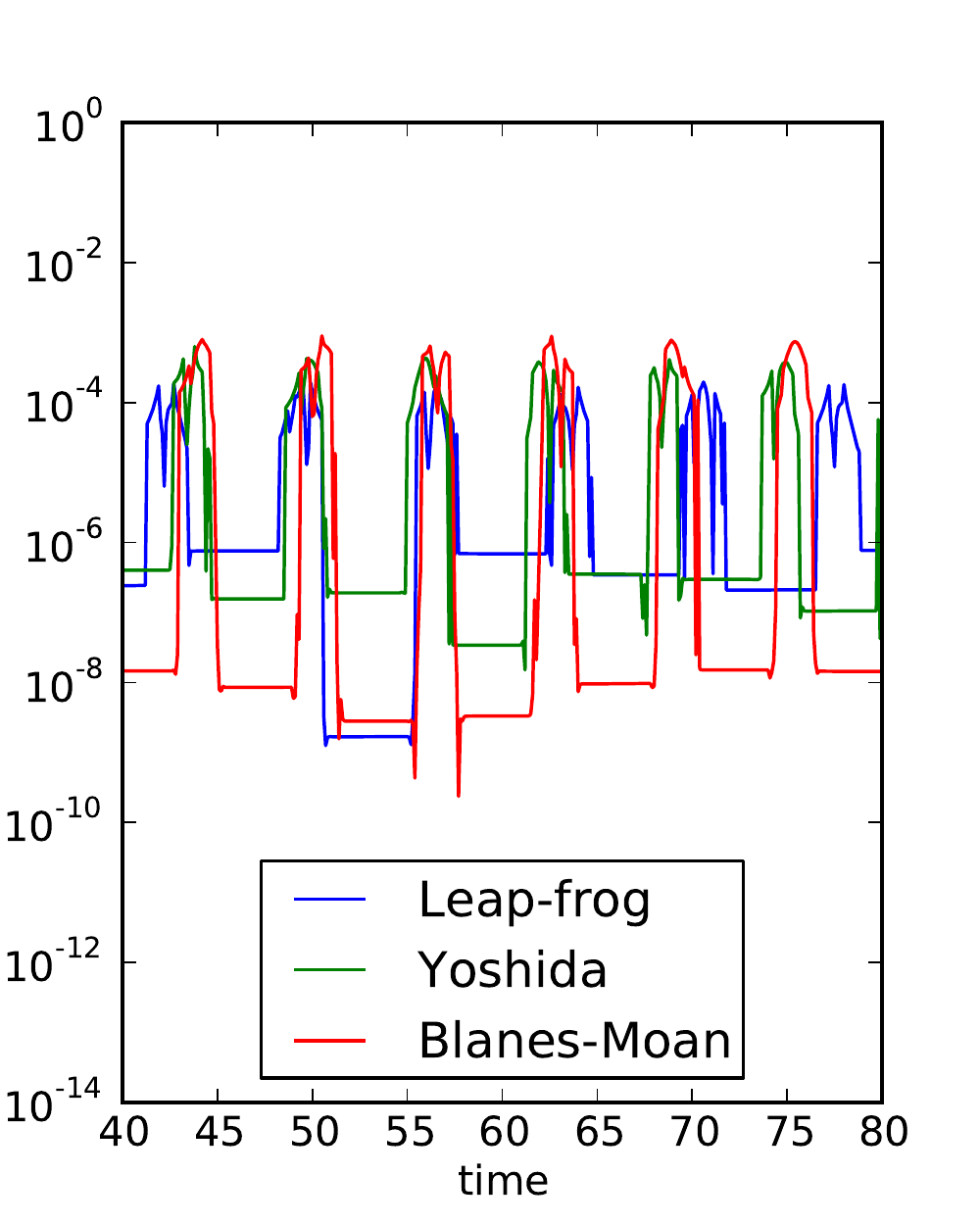}
  \caption{The difference between modified energy at a given time and
    the initial modified energy for the Kepler method with
    eccentricity $ecc=0.6$. The left plot shows the results for the
    St\"ormer--Verlet method for different step sizes, while the right
    plot shows the results for different methods. All methods are run
    with step size~$h=0.1$, so St\"ormer--Verlet does considerably
    less work.}
  \label{figkepler2}
\end{figure}

The left plot in figure~\ref{figkepler2} illustrates for several
different step sizes how the modified energy varies. There are peaks
when the particle is near the singularity at the origin. The crucial
point is that the energy essentially recovers its value after this
point before another close encounter. The plot on the right compares
the three different methods. It is seen that the methods give rather
different results, even though the maximal variation in the modified
energy is almost the same for the methods. The Blanes--Moan method
seems to preserve the modified energy better after the close
encounter, which might indicate a special advantage of this method
when applied to the Kepler problem. 
If, however, the time steps are scaled so that the computational cost
is the same for the three methods, the St\"ormer--Verlet method will
preserve the modified energy better than the high-order methods.

\begin{figure}
  \centering
  \includegraphics[width=12cm]{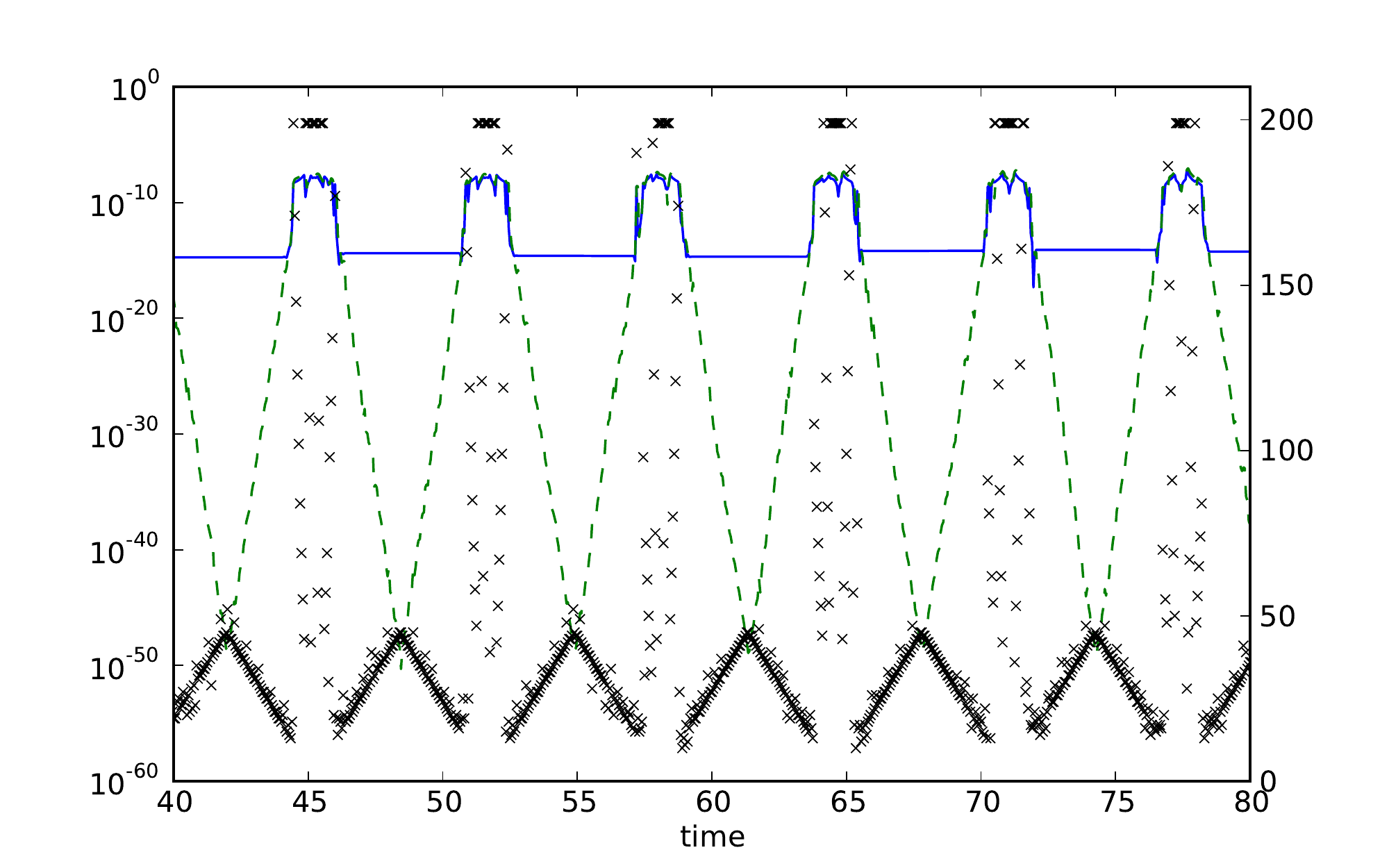}
  \caption{The change per step in the modified energy (dashed) and the
    accumlated change (solid), and the optimal order~$m$ (capped
    by~200, marked by 'x'). The axis for the energies is on the left, which the right axis is for the order~$m$. This is for the St\"ormer--Verlet method
    applied to the Kepler problem with $h=1/20$ and $ecc=0.6$.}
  \label{figkepler3} 
\end{figure} 

Figure~\ref{figkepler3} shows the accumulated change in energy,
$|\overline H(p_0,q_0,t_0) - \overline H(p_n,q_n,t_n)|$, and the
instantaneous change in energy, $|\overline H(p_{n-1},q_{n-1},t_{n-1})
- \overline H(p_n,q_n,t_n)|$, together with the optimal~$m$ found by
the error estimate~\eqref{errest}.  The graph shows that near the
singularity quite a high order~$m$ (which we bounded by 200) is
used. This indicates that information from the smooth parts of the
trajectory is used near singularities, and that it might be important
to use very high order approximations to get a clear picture of the
drift. The graph also indicates that the algorithm can track
instantaneous changes in energy.

Away from the parts of the orbit where the singularity at the origin
is approached most closely, a lower value of~$m$ suffices. It is thus
useful to find a more efficient method for finding the optimal~$m$
instead of computing the error estimate for all~$m$ up to some large
value (here,~200). We found good results with the following ad-hoc
termination criterion: compute the error estimate~\eqref{errest} for
all~$m$ up to the first value of~$m$ for which
\[
\max_{j=m-11,\dots,m-1} |T_{j,j} - T_{j-1,j-1}| \le 
\max_{j=m-10,\dots,m} |T_{j,j} - T_{j-1,j-1}|,
\]
and then choose the~$m$ with the minimal error estimate. The plots
produced by this criterion are nearly indistinguishable from the plots
produced when all~$m$ up to~200 are considered.

\subsection{H\'enon--Heiles system}

\begin{figure}
  \includegraphics[width=0.5\linewidth]{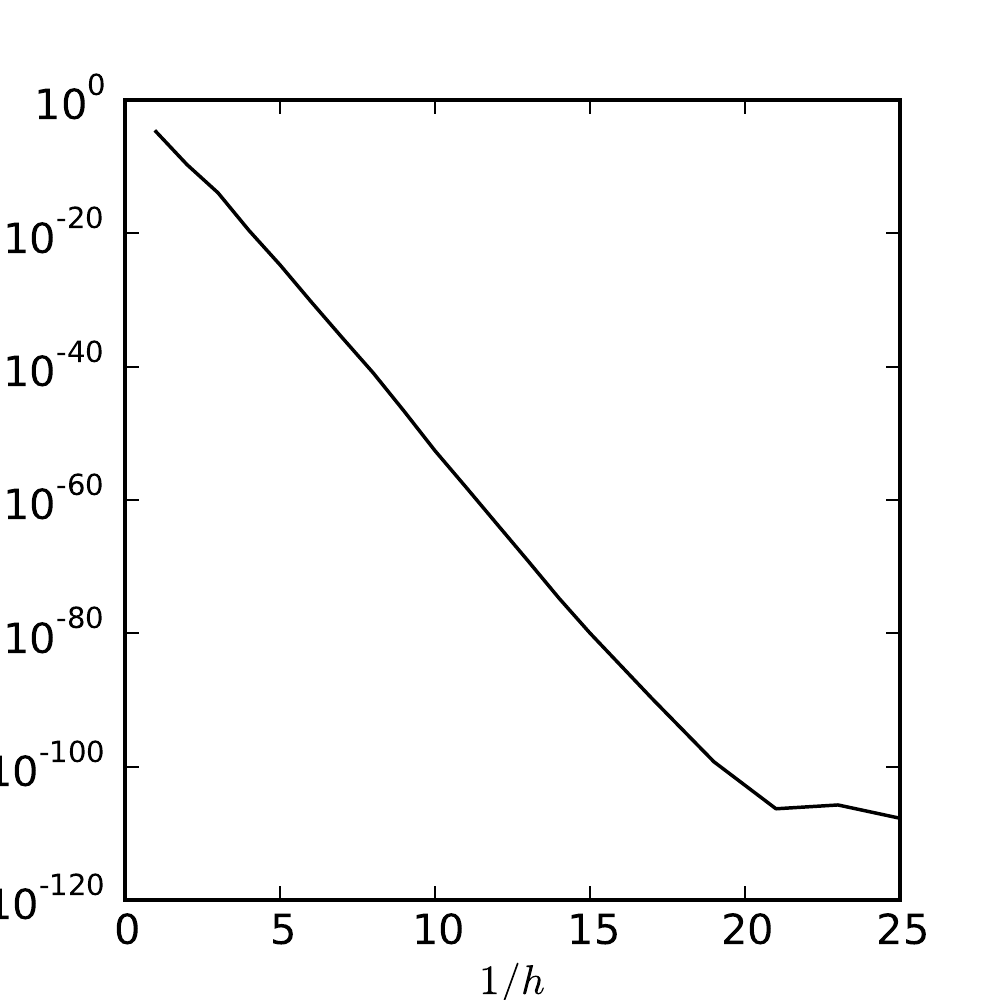}%
  \includegraphics[width=0.5\linewidth]{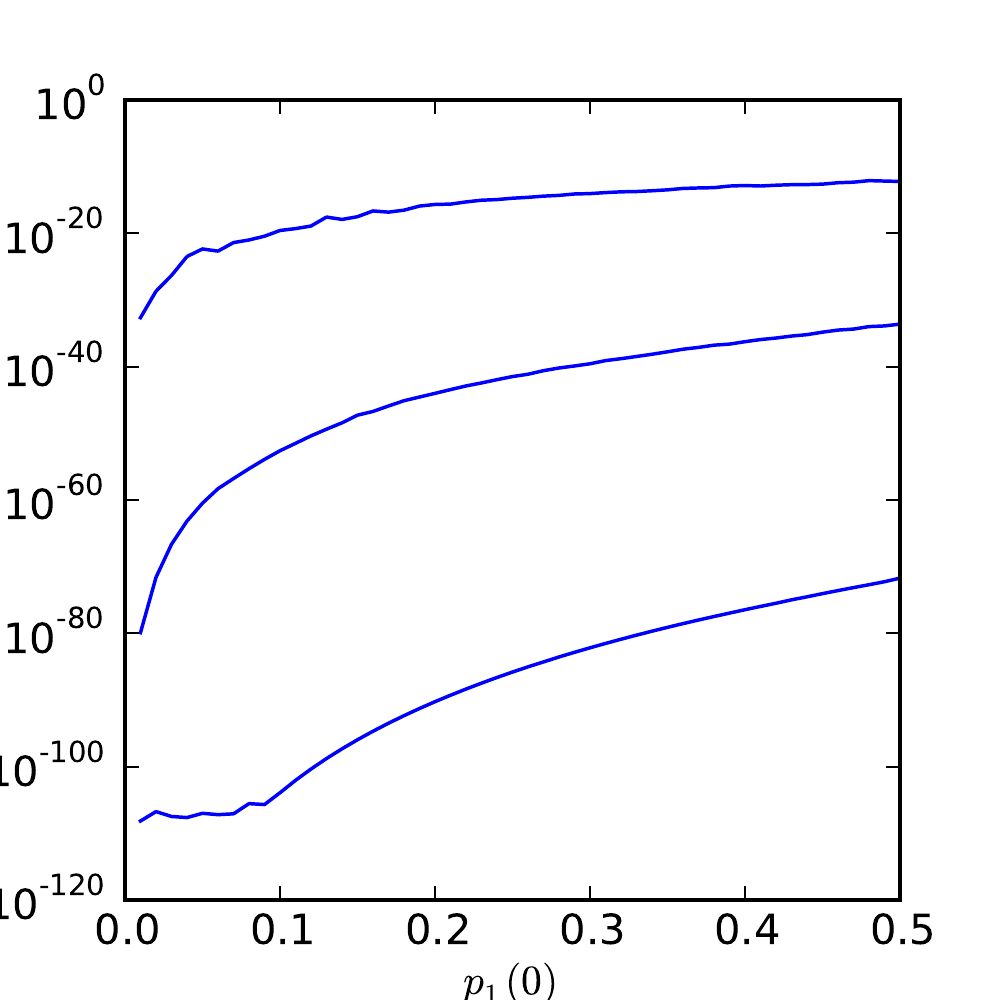}
  \caption{Drift in modified energy for the St\"ormer--Verlet method
    applied to the H\'enon--Heiles problem as a function of step
    size~$h$ with initial condition $p_1(0) = 0.1$ (left plot), and as
    a function of~$p_1$ for step sizes $h=0.25, 0.1, 0.05$ (right
    plot).}
  \label{hh1}
\end{figure}

The Hamiltonian of the H\'enon--Heiles system is given by
\[
H = \frac12 (p_1^2 + p_2^2) + 
\frac12 (q_1^2 + q_2^2 + 2q_1^2q_2 - \tfrac23 q_2^3). 
\]
Skeel \textsl{et al.}\ investigated the theoretical $\exp(-c/h)$
behaviour of the drift in the modified energy for this system, and
report that the results are ``less convincing''~\cite[\S2.4]{engle05med}.
We revisit this problem, using instead the extrapolation method to
achieve arbitrary high orders. Figure~\ref{hh1} shows that the
expected exponential smallness is indeed present, and that there is no
problem in using the algorithm other than allowing for large values
of~$m$ (we again capped~$m$ at~200). The effect of round-off error
becomes visible when $1/h$ exceeds~20; remember that all computations
are done with 120~digits.

The right plot shows how the maximal deviation of the modified energy
changes as the initial condition for $p_1$ is varied; the initial
conditions for the other variables are fixed as $q_1 = q_2 = p_2 = 0$.
This plot shows that there is no abrupt change in energy preservation
when moving from regular, integrable motions (the region with energy
$H < 1/12$ or, equivalently, $p_1 < 1/\sqrt{6}\approx 0.4$) to the
chaotic regime of phase space (where $H > 1/12$).

We also applied the fourth-order methods due to Yoshida and Blanes and
Moan to this problem, with the same results as for the Kepler
problems: the St\"ormer--Verlet method shows slightly better energy
preservation, but the value of~$c$ in the $\exp(-c/h)$ dependence is
the same.

\section{Conclusions}

We have supplied rigorous estimates for a numerical algorithm that
computes the modified energy for methods based on operator splitting
of Hamiltonian systems. The estimate shows that the procedure can
recover exponentially small estimates, known to exist
theoretically. The estimates exhibit the same dependence on the
important parameters $\tilde r_y$, $\tst$ and $\|f\|_\domD$, and can
therefore in principle be used to extract their values from
simulations. When comparing different splitting algorithms, it seems
that in the limit $h\rightarrow 0$ the exponential remainder term only
weakly depends on the method coefficients. Thus when considering the
additional cost of optimized, many-stage, methods these will have a
larger drift than the second-order St\"ormer--Verlet algorithm. In
other words, when it comes to preserving the modified energy, cheap,
low-order methods are preferable. Although we have not considered ODEs
originating from Hamiltonian semidiscretization of PDEs it seems
likely that for long time simulations a low-order method such as
St\"ormer--Verlet is preferable if energy preservation is important.

\section*{Appendix}

\begin{proof}[Proof of Lemma \ref{extlem}]
Without loss of generality we assume that $n=0$.

By representing (\ref{cdiff}) as a contour integral we have
\[
D_my(0) = \frac 1{2\pi i} \sum_{j=1}^m \oint_{\gamma}\frac{(-1)^{j+1} (m!)^2}{\tst j(m-j)!(m+j)!} \left\{ \frac{1}{z-j\tst}-\frac{1}{z+j\tst} \right\} y(z) dz,
\]
where the contour $\gamma$ includes the points $-m\tst,\ldots,m\tst$
and excludes singularities of~$y$, as sketched in Figure~\ref{extlem}.

\begin{figure}
\psfrag{r1}{{\small $\rho$}}
\psfrag{r2}{{\small$\rho$}}
\psfrag{r3}{{\small$\rho$}}
\psfrag{co}{{\small$\gamma$}}
\psfrag{m1}{{\small$-mh$}}
\psfrag{m2}{{\small$-(m-1)\tst$}}
\psfrag{m3}{{\small$-(m-2)\tst$}}
\psfrag{m7}{{\small$-\tst$}}
\psfrag{m8}{{\small$\tst$}}
\psfrag{m4}{{\small$(m-2)\tst$}}
\psfrag{m5}{{\small$(m-1)\tst$}}
\psfrag{m6}{{\small$m\tst$}}
\psfrag{Im}{{\small${\Im}(z)$}}
\psfrag{Re}{{\small${\Re}(z)$}}
\centerline{\includegraphics[width=375pt]{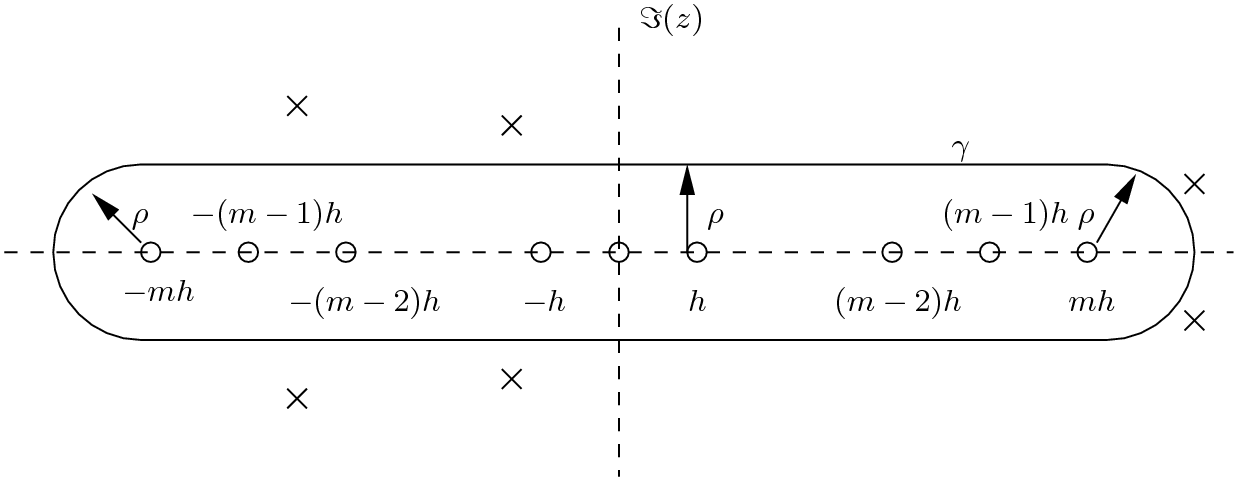}}
\caption{The contour of integration used in the proof of Lemma~\ref{extlem}.}
\end{figure}

The derivative is given by $y'(0) = \frac 1{2\pi i} \oint_\gamma
\frac{y(z)dz}{z^2}$, so the error in the approximation becomes
\begin{equation}
E_m(y)(0) = D_my(0)-y'(0)=\frac 1{2\pi i} \oint_{\gamma} K_m(z)y(z)dz,
\label{errorci}
\end{equation}
where the kernel is defined by
\[
K_m(z) = \frac{(-1)^{m+1} (m!)^2 \tst^{2m}}{z^2 (z^2-\tst^2)(z^2-(2\tst)^2)\cdots (z^2-(m\tst)^2)}.
\]
Along the curve~$\gamma$, the kernel~$K_m$ achieves its maximum in
modulus at $z=i\rho$, and the maximum is
\begin{align*}
  |K_m(i\rho)| &= \frac{(m!)^2 \tst^{2m}}{\rho^2(\rho^2+\tst^2)\cdots(\rho^2+(m\tst)^2)}=\frac{1}{\rho^2 (1+\frac{\rho^2}{\tst^2})\cdots(1+\frac{\rho^2}{(m\tst)^2})} \\
  &= \frac{\pi }{\rho\tst \sinh(\frac{\pi \rho}{\tst})} \prod_{j=m+1}^\infty \left(1+\frac{\rho^2}{(\tst j)^2}\right)
\end{align*}
where the last equality follows from $\prod_{j=1}^\infty
(1+\frac{\rho^2}{(\tst j)^2}) = \frac{\tst}{\pi\rho} \sinh(\pi
\rho/\tst)$.  The product can be bounded as 
\begin{align*}
  \log \prod_{j=m+1}^\infty \left(1+\frac{\rho^2}{(\tst j)^2}\right)
  &= \sum_{j=m+1}^\infty \log\left(1+\frac{\rho^2}{(\tst j)^2}\right) \\
  &\leq \int_{m}^\infty \log\left(1+\frac{\rho^2}{(\tst x)^2}\right)dx
  \leq \frac{\rho^2}{\tst^2 m},
\end{align*}
yielding $\prod_{j=m+1}^\infty (1+\frac{\rho^2}{(\tst j)^2}) \leq
\exp(\frac{\rho^2}{\tst^2 m})$, and thus
\[
|K_m(i\rho)| \leq \frac{\pi }{\rho\tst \sinh(\frac{\pi \rho}{\tst})} \exp\left(\frac{\rho^2}{\tst^2 m}\right).
\]
Since the length of the contour is $2\pi \rho + 4m\tst$, the error
expression~\eqref{errorci} can be bounded as
\[
|D_m y(0)-y'(0)| \leq \frac { (\pi\rho+2m\tst)\exp\left(\frac{\rho^2}{\tst^2 m}\right)}{\rho \tst\sinh\left(\frac{\pi \rho}{\tst}\right)} \|y\|_{\rho} 
\]
where $\|y\|_{\rho}:=\sup_{|\Im(z)| < \rho} |y(z)|_\infty$. This upper
bound is minimized by choosing $m$ so that $\frac{d}{dm}
(\pi\rho+2m\tst) \exp(\rho^2/\tst^2m)$ vanishes, i.e. $m \approx
\frac{\rho^2}{\tst^2}$. This gives the bound
\[
|D_my(0)-y'(0)| \leq \frac{ e (\pi+2\rho^2/\tst^2) }{\tst\sinh\left(\frac{\pi\rho}{\tst}\right)}\|y\|_\rho \leq C_1 \frac{\rho \exp\left(-\frac{\pi \rho}{\tst}\right)}{\tst^2}\|y\|_\rho
\] 
for some constant $C_1>0$.
\end{proof}

\begin{proof}[Proof of Corollary \ref{coro1}]
This follows from
\begin{align*}
  &|\overline H(p_n,q_n,t_n)-T^n_{m^*,m^*}| \\
  &\qquad \leq \tfrac12 |\overline q^T(\overline p'-D_{m^*}\overline p)|
  + \tfrac12 |\overline p^T(\overline q'-D_{m^*}\overline q)|
  + \tfrac12 |\overline \beta'-D_{m^*}\overline \beta| \\
  &\qquad \leq \tfrac12 |q_n|_1 \|\overline p'-D_{m^*}\overline p\|_\rho
  + \tfrac12 |p_n|_1 \|\overline q'-D_{m^*}\overline q\|_\rho
  + \tfrac12 \|\overline \beta'-D_{m^*}\overline \beta\|_\rho \\
  &\qquad \leq \frac{C_1\rho \exp(-\frac {\pi \rho}h)}{2h^2} \,
  (|q_n|_1\|\overline p\|_\rho + |p_n|_1\|\overline q\|_\rho
  + \|\overline \beta\|_\rho)
\end{align*}
where the last inequality follows from Lemma \ref{extlem}.
\end{proof}

\begin{proof}[Proof of Lemma \ref{lemmaana}]
We prove this by Picard iteration: set $\tilde x_1 = x_n$ and iterate 
$\tilde x_{k+1}(t) = x_n + \int_0^t g(\tilde x_k(s),s) \, ds$. Fix
$t\in \domD_t$, and assume at first that $r_t$ is sufficiently
large. For $\tilde x_{k+1},\tilde x_k \in \domD_y$
\begin{align*}
  &|g(\tilde x_{k+1},t)-g(\tilde x_k,t)|_\infty \\
  &\qquad = \left| \int_0^1 \frac{d}{ds}
  g  \Bigl(\tilde x_{k+1}+s(\tilde x_k-\tilde x_{k+1}),t \Bigr) 
  \, ds \right|_\infty \\
  &\qquad = \frac1{2\pi} \left| \int_0^1 \oint_{|z-s|=R} 
  \frac{g\bigl(\tilde x_{k+1}+z(\tilde x_k-\tilde x_{k+1}),t\bigr)}{(z-s)^2} 
  \, dz \, ds \right|_\infty \\
  &\qquad = \frac1{2\pi} \left| \int_0^1 \oint_{|w|=R} 
  \frac{g\bigl(\tilde x_{k+1}+s(\tilde x_k-\tilde x_{k+1})+w(\tilde x_k-\tilde x_{k+1}),t\bigr)}{w^2}
  \, dw \, ds \right|_\infty.  
\end{align*}
The radius $R$ is restricted by the requirement that the argument of
$g$ lies within $\domD_y$ or
\begin{align*}
  &|\tilde x_{k+1} + s(\tilde x_k-\tilde x_{k+1}) - x_n
  + w(\tilde x_k-\tilde x_{k+1})|_\infty \\
  &\qquad \leq |\tilde x_{k+1} + s(\tilde x_k-\tilde x_{k+1})-x_n|_\infty
  + r|(\tilde x_k-\tilde x_{k+1}))|_\infty \\
  &\qquad \leq |t|\|g\|_r +R|\tilde x_{k}-\tilde x_{k+1}|_\infty < r_y 
\end{align*}
by using 
\begin{multline*}
  |\tilde x_{k+1}+s(\tilde x_k-\tilde x_{k+1})-x_n)|_\infty \\
  \leq \sup_{|\tau|=|t|} \Bigl| 
  \int_0^{\tau} (1-s) \, |g(\tilde x_{k}(\tau),\tau)|_\infty +
  s \, |g(\tilde x_{k-1}(\tau),\tau)|_\infty \,d\tau \Bigr| 
  \leq |t|\|g\|_r.
\end{multline*}
We may therefore choose
\[ 
R = \eta \frac{r_y-|t|\|g\|_r}{|\tilde x_{k+1}-\tilde x_k|_\infty},
\quad 0<\eta<1
\] 
which gives the supremum-norm Lipschitz bound 
\[
|g(\tilde x_{k+1},t)-g(\tilde x_k,t)|_\infty \leq
\frac{\|g\|_r}{\eta(r_y-|t|\|g\|_r)}|\tilde x_{k+1}-\tilde x_k|_\infty.
\]
Let $\Delta_{k+1}(|t|) = \sup_{|\tau|=|t|}
|\tilde x_{k+1}(\tau)-\tilde x_k(\tau)|_\infty$, then the Picard iteration
$\tilde x_1 = x_n$, $\tilde x_{k+1}=x_n + \int_0^t g(\tilde x_k(s),s)
\, ds$, converges if $\Delta_k\rightarrow 0$ as $k\rightarrow\infty$, with
\[
\Delta_{k+1}(t) \leq \int_0^{|t|}
\frac{\|g\|_r}{\eta(r_y-s\|g\|_r)} \Delta_k(s) \, ds,\quad
\Delta_1(t)=|t|\|g\|_r 
\]
Introducing the generating function $G(\mu,|t|)=\sum_{k\ge1} \mu^k
\Delta_k$ we have 
\[
G(\mu,t) \leq \mu |t|\|g\|_r + \mu \int_0^{|t|}
\frac{\|g\|_r}{\eta(r_y-s\|f\|_\rho)} G(\mu,s) \, ds. 
\]
Since the terms in
this inequality are positive, an upper bound is the solution of
\[
\frac{dG^+(\mu,|t|)}{d|t|} = \mu \|g\|_r +
\frac{\|g\|_r}{\eta(r_y-|t|\|g\|_r)} G^+(\mu,|t|), \qquad
G^+(\mu,|t|=0)=0,
\]
i.e.
\[
G^+(\mu,|t|) = \frac{\mu\eta\rho}{\eta+\mu} \biggl(
\Bigl( 1 - \frac{|t|\|g\|_r}{r_y} \Bigr)^{-\mu/\eta}
- \Bigl( 1 - \frac{|t|\|g\|_r}{r_y} \Bigr) \biggr).
\]
Because $G^+$ is analytic in $\mu$ around $\mu=1$ provided
$|t|<\frac{r_y }{\|g\|_r}$, the sequence $\Delta_k(|t|)$ converges
uniformly to zero and hence $\tilde x_k(t)$ converges \emph{uniformly}
to the solution. Since each iterate $\tilde x_{k+1}(t) = x_n +
\int_0^t f(\tilde x_k(s),s) \, ds$ is analytic in $t\in \{t \in \CC:
|t| < \min\{\frac{r_y}{\|g\|_r},r_t\}\}$ the uniform convergence gives
by Weierstrass theorem that $y(t) = \tilde x_\infty(t)$ is analytic in
this domain as well.
\end{proof}

\begin{proof}[Proof of Theorem \ref{theomod}]
In Theorem \ref{BEA} we take $g=\overline f$, thus $\|g\|_r \leq 
\frac{2}{1-\eta}\|f\|_\domD$ with $\overline f$ analytic in $\overline
\domD'$. This gives that the $\overline y(t)$ is analytic in the
domain
\[
\left\{ \tau \in \CC : |\Im(\tau)| < 
\min \left( \frac{\tilde r_y-\delta}{\frac{2}{1-\eta}\|f\|_\domD},
  \frac{\eta \delta}{e\|f\|_\domD} \right) \right\}.
\]
We find that the bound is optimized by picking $\tilde r_y=\eta\delta$ where $\eta=\frac{e-2}{e-\sqrt{2e}}$. Thus we may take $\rho=\frac{\eta \delta}{e\|f\|_\domD}< 0.6835 \frac{\delta}{\|f\|_\domD}$ in Corollary \ref{coro1}
giving the exponentially small bound ($t=n\tst$)
\[
|\overline H(p_n,q_n,t)-T_{m^*,m^*}| 
\leq \frac{C_1}{2\tst^2} 
\exp\left(-C_2 \frac{\delta}{\tst \|f\|_\domD}\right)
(|q_n|_1 \|\overline p\|_\rho + |p_n|_1 \|\overline q\|_\rho
+ \|\overline \beta\|_\rho),
\] 
where $C_2=0.6834 \pi$. 
\end{proof}

\bibliographystyle{abbrv}
\bibliography{../jitse}

\end{document}